\documentclass[a4paper]{article}
\usepackage[affil-it]{authblk}

\usepackage[english]{babel}
\usepackage[utf8]{inputenc}
\usepackage{amsmath}
\usepackage{amssymb}
\usepackage{graphicx}
\usepackage{enumerate}
\usepackage{amsthm}

\newtheorem{theorem}{Theorem}[section]
\newtheorem{proposition}[theorem]{Proposition}
\newtheorem{lemma}[theorem]{Lemma}
\newtheorem{corollary}[theorem]{Corollary}

\newtheorem{definition}[theorem]{Definition}

\theoremstyle{remark}

\newtheorem*{claim}{Claim}
\newtheorem*{remark}{Remark}

\newcommand{\sigmadet}[1]{(\mathbf{\Sigma^0_2})_{#1}\textrm{-}\mathrm{Det}}
\newcommand{\ouraxiom}{\forall x[(\mathbf{\Sigma^0_2})_x\textrm{-}\mathrm{Det}]}
\newcommand{\ouraxiomcantor}{\forall x[(\mathbf{\Sigma^0_2})_x\textrm{-}\mathrm{Det}^*]}

\newcommand{\ca}{\textrm{-}\mathsf{CA}_0}
\newcommand{\ind}{\textrm{-}\mathsf{IND}}
\newcommand{\aca}{\mathsf{ACA}_0}

\newcommand{\wkl}{\mathsf{WKL}_0}

\newcommand{\bbn}{\mathbb{N}}

\newcommand{\tuple}[1]{\langle #1 \rangle}

\newcommand{\as}[2]{\forall #1 \! < \! #2 \,}
\newcommand{\eb}[2]{\exists #1 \! \le \! #2 \,}

\newcommand{\Aa}{{\mathcal A}}
\newcommand{\Bb}{{\mathcal B}}

\newcommand{\Dd}{{\mathcal D}}

\newcommand{\Pp}{{\mathcal P}}

\newcommand{\Ss}{{\mathcal S}}

\usepackage{pgf}
\usepackage{tikz}
\usepackage{tikz-qtree}
\usetikzlibrary{arrows,automata,shapes.geometric,arrows,calc,decorations.pathmorphing,positioning,shadows,fit,shapes}
\usepackage{wrapfig}
\usepackage[font={scriptsize,it}]{caption}

\usepackage{float}

\title{How unprovable is Rabin's decidability theorem?}

\author{Leszek Aleksander Ko{\l}odziejczyk\thanks{Partially supported by Polish National Science Centre grant no. 2013/09/B/ST1/04390.} ~and Henryk Michalewski\thanks{Partially supported by Polish National Science Centre grant no.~2014/13/B/ST6/03595}}
\affil{Institute of Mathematics, University of Warsaw, Poland}

\date{\today}

\begin{document}
\maketitle

\begin{abstract}
We study the strength of set-theoretic axioms needed to prove Rabin's theorem on the decidability of the MSO theory of the infinite binary tree. We first show that the complementation theorem for tree automata, which forms the technical core of typical 
proofs of Rabin's theorem, is equivalent over the moderately strong second-order arithmetic theory $\mathsf{ACA}_0$ to a determinacy principle implied by the positional determinacy of all parity games and implying the determinacy of all Gale-Stewart games given by boolean combinations of ${\bf \Sigma^0_2}$ sets. It follows that complementation for tree automata is provable from $\Pi^1_3$- but not $\Delta^1_3$-comprehension.

We then use results due to MedSalem-Tanaka, M\"ollerfeld and Heinatsch-M\"ollerfeld to prove that over $\Pi^1_2$-comprehension, the complementation theorem for tree automata, decidability of the MSO theory of the infinite binary tree, positional determinacy of  parity games and determinacy of $\mathrm{Bool}({\bf \Sigma^0_2})$ Gale-Stewart games are all equivalent. Moreover, these statements are equivalent to the $\Pi^1_3$-reflection principle for $\Pi^1_2$-comprehension. It follows in particular that Rabin's decidability theorem is not provable in $\Delta^1_3$-comprehension.
\end{abstract}

\section{Introduction}

Rabin's decidability theorem \cite{rabin_dec} says that the monadic second order (MSO) theory of the infinite binary tree $\{0,1\}^*$ with the left and right successor relations is decidable. In the words of the  book ``The Classical Decision Problem'' \cite{gradel}, Rabin's result is ``one of the most important decidability theorems for mathematical theories and has numerous applications in several areas of mathematics and computer science''  (for a discussion, see e.g.\ \cite[Chapter 7]{gradel}).

Unlike other prominent decidability results, such as the ones for Presburger arithmetic, real-closed fields or even the MSO theory of $(\bbn, \le)$, Rabin's theorem appears likely to involve significant logical strength, in the sense of being unprovable without axioms asserting the existence of very abstract and complicated sets. This is exemplified by the fact that typical modern proofs of Rabin's theorem, dating back to \cite{gurevich_harrington},  invoke the determinacy of certain infinite games with Borel winning conditions, and determinacy principles are notorious for requiring very large logical strength.

The framework of \emph{reverse mathematics} (see \cite{simpson})
offers a natural way of measuring the logical strength of a theorem. The idea is that many mathematical theorems can be formalized in the language of \emph{second-order arithmetic}, a foundational axiomatic theory used already by Hilbert and Bernays. The most important axiom of second-order arithmetic is the comprehension scheme, stating the existence of any set of natural numbers defined by first- and second-order quantification over $\bbn$. Reverse mathematics proceeds by analyzing various mathematical statements and proving their equivalence, over a suitable weak base theory, to some rather limited form of comprehension. Most mathematical theorems analyzed in this fashion have turned out to require no more than the theory $\aca$ allowing only \emph{arithmetical} comprehension, that is, the existence of sets defined without any second-order quantifiers. Theorems requiring strictly more than $\Pi^1_1$-comprehension, or the existence of sets defined in terms of one second-order quantifier, are quite exceptional. Such exceptions include the celebrated graph minor theorem \cite{friedman_robertson_seymour}, and a theorem in general topology proved equivalent to $\Pi^1_2$-comprehension in \cite{mummert_simpson}.

Determinacy theorems are a more extreme exception. $\Pi^1_1$-comprehension is barely enough to prove the determinacy of games in which the winning condition is the intersection of an open set and a closed set. $\Pi^1_2$-comprehension proves that $F_\sigma$ games are determined \cite{tanaka}, but can no longer do so for games given by arbitrary boolean combinations of $F_\sigma$ sets (essentially\footnote{The original reasoning presented in \cite{medsalem_tanaka_delta_03} applies to $\mathbf{\Delta^0_3}$ sets, but it is easily adaptable to the case of boolean combinations of $F_\sigma$ sets. We present this argument in Section \ref{sec:determinacy-comprehension}. } \cite{medsalem_tanaka_delta_03}). For arbitrary boolean combinations of $F_{\sigma\delta}$ sets, proving determinacy requires going beyond second-order arithmetic, an immensely strong theory in most other respects \cite{montalban_shore}. 

The aim of this paper is to analyze, reverse mathematics--style, the logical strength of Rabin's theorem. We first focus our attention 
on the complementation theorem for automata on infinite trees, which is the key ingredient in typical proofs 
of Rabin's theorem and is often simply considered an alternative formulation of the decidability result. We prove that the complementation theorem is equivalent over $\aca$ to a determinacy principle intermediate
between the determinacy of games given by arbitrary boolean combinations of $F_\sigma$ sets and the slightly stronger statement that all parity games \cite{emerson_jutla} are positionally determined\footnote{The idea of considering the parity acceptance condition for automata in order to prove Rabin's decidability theorem was first proposed in \cite{emerson_jutla}.}. Using earlier work on determinacy of Gale-Stewart games, we conclude that the complementation theorem is provable in $\Pi^1_3$- but not $\Pi^1_2$- or even $\Delta^1_3$-comprehension.

We then consider Rabin's decidability result itself. Using the work of \cite{medsalem_tanaka_delta_03}, we prove that over $\Pi^1_2$-comprehension, already the statement ``the $\Pi^1_3$ fragment of the MSO theory of $\{0,1\}^*$ is decidable'' implies determinacy of arbitrary boolean combinations of $F_\sigma$ sets, which makes it unprovable in $\Delta^1_3$-comprehension. On the other hand, any version of the decidability theorem that can be stated in second-order arithmetic follows from the positional determinacy of parity games, and thus also from $\Pi^1_3$-comprehension.

The final part of our work relies on techniques developed in Michael M\"ollerfeld's PhD thesis \cite{mollerfeld_thesis}, which links $\Pi^1_2$-comprehension with an arithmetical version of the $\mu$-calculus. Using a slight strengthening of  M\"ollerfeld's results, we show that over $\Pi^1_2$-comprehension, the complementation theorem for automata, Rabin's decidability theorem, and all the main determinacy principles studied in the paper are actually equivalent. Moreover, all these statements are equivalent to a purely logical \emph{reflection principle}: ``all $\Pi^1_3$ sentences provable using $\Pi^1_2$-comprehension are true'' (note that the unprovablity of this priciple in $\Pi^1_2$-comprehension follows from G\"odel's second incompleteness theorem).

Our paper is structured as follows. Section \ref{sec:preliminaries} presents the necessary background in reverse mathematics, automata theory and games. In Section \ref{sec:complementation}, we characterize the logical strength of the complementation theorem for tree automata in terms of a determinacy principle. We review some earlier results on determinacy principles in Section \ref{sec:determinacy-comprehension}, and then use them to analyze the strength of Rabin's decidability theorem in Section \ref{sec:decidability}. Our final result linking complementation for automata, decidability of MSO and various determinacy statements with a reflection principle for $\Pi^1_2$-comprehension is discussed in Section \ref{sec:reflection}.


\section{Basic notions}\label{sec:preliminaries}

\subsection{Second-order arithmetic} \emph{Second-order arithmetic} is a natural framework for studying the strength of axioms needed to prove theorems of countable mathematics, that is, the part of mathematics concerned with objects that can be represented using no more than countably many bits of information. The two-sorted language of second-order arithmetic, $L_2$, contains \emph{first-order} variables $x, y, z, \ldots$, intended to range over natural numbers, and \emph{second-order} variables $X,Y,Z,\ldots$, intended to range over sets of natural numbers. $L_2$ includes the usual arithmetic functions and relations $+, \cdot, \le, 0,1$ on the first-order sort, and the $\in$ relation which has one first-order and one second-order argument. 

\emph{Full second-order arithmetic}, $\mathsf{Z}_2$, has axioms of three types: (i) axioms stating that the first-order sort is the non-negative part of a discretely ordered ring; (ii) comprehension axioms, or sentences of the form
\[\forall \bar Y \, \forall \bar y\, \exists X\, \forall x\, (x\in X \Leftrightarrow \varphi(x, \bar Y, \bar y)),\]
where $\varphi$ is an arbitrary formula of $L_2$ not containing the variable $X$; (iii) the induction axiom,
\[\forall X\, \left[0 \in X \land \forall x\, (x \in X \Rightarrow x+1 \in X)\Rightarrow \forall x\,(x\in X) \right].\]

The language $L_2$ is surprisingly expressive, as the first-order sort can be used to encode arbitrary finite objects and the second-order sort can encode even such objects as complete separable metric spaces, continuous functions between them, and Borel sets within them (cf.\ \cite[Chapters II.5, II.6, V.3]{simpson}). Moreover, the theory $\mathsf{Z}_2$ is very strong: almost all theorems from a typical undergraduate course that are expressible in $L_2$ can be proved 
in $\mathsf{Z}_2$. In fact, the basic observation underlying the programme of reverse mathematics \cite{simpson} is that  many important theorems are equivalent to various fragments of $\mathsf{Z}_2$, where the equivalence is proved in some specific weaker fragment, referred to as the \emph{base theory}. 

For relatively strong theorems such as the ones studied in this paper, a reasonable base theory is $\aca$, the fragment of $\mathsf{Z}_2$ obtained by restricting the comprehension scheme to instances where the formula $\varphi$ is \emph{arithmetical}, that is, does not contain second-order quantifiers (second-order free variables are allowed). Stronger fragments can be obtained by allowing comprehension for $\varphi$ with a fixed number of second-order quantifiers. A formula is $\Pi^1_n$ if it has the form
\[\forall X_1\, \exists X_2 \, \ldots \mathsf{Q} X_n\, \psi\]
with $\psi$ arithmetical; it is $\Sigma^1_n$ if it is the negation of a  $\Pi^1_n$ formula; it is $\Delta^1_n$ if it is equivalent to both a $\Pi^1_n$ and a $\Sigma^1_n$ formula. The theory $\Pi^1_n\ca$ is obtained by restricting the comprehension scheme to instances where $\varphi$ is $\Pi^1_n$.  In the subtheory $\Delta^1_n\ca$ of $\Pi^1_n\ca$, the comprehension scheme takes the form
\[\forall \bar Y \, \forall \bar y\, \left[\forall x\, (\varphi(x,\bar Y, \bar y) \Leftrightarrow \neg\psi(x,\bar Y, \bar y)) \Rightarrow \exists X\, \forall x\, (x\in X \Leftrightarrow \varphi(x, \bar Y, \bar y)) \right],\]
where both $\varphi$ and $\psi$ are $\Pi^1_n$.

The induction scheme $\Sigma^1_n\textrm{-}\mathsf{IND}$ consists of the sentences
\[\forall \bar Y \, \forall \bar y\, \left[\varphi(0,\bar Y, \bar y) \land \forall x\, (\varphi(x,\bar Y, \bar y) \Rightarrow \varphi(x+1,\bar Y, \bar y)) \Rightarrow \forall x\,\varphi(x,\bar Y, \bar y) \right] \]
where $\varphi$ is $\Pi^1_n$ or $\Sigma^1_n$. By the induction and comprehension axioms, $\Pi^1_n\ca$ proves $\Sigma^1_n\textrm{-}\mathsf{IND}$. On the other hand, $\Pi^1_n\ca$ does not prove 
$\Sigma^1_{n+1}\textrm{-}\mathsf{IND}$, while $\bigcup_{n \!\in \!\omega} \Sigma^1_{n}\textrm{-}\mathsf{IND}$ does not prove $\Pi^1_{n}\ca$ even assuming $\Pi^1_{n-1}\ca$ (cf.\ \cite{simpson}).

 In this paper, the most prominent theory is $\Pi^1_2\ca$. We present some important principles provable in $\Pi^1_2\ca$. The first two principles are related to countable sequences of sets. Such a sequence $\tuple{X_i}_{i \in \bbn}$ can be represented by a single set $X$ if we let
\[x \in X_i \Leftrightarrow \tuple{i,x} \in X,\]
where $\tuple{\cdot,\cdot}$ is some standard pairing function. We write $X \in \tuple{X_i}_{i \in \bbn}$ if there is some $i$ such that $X = X_i$.

\begin{definition}
The \emph{$\Sigma^1_2$ axiom of choice}, $\Sigma^1_2\textrm{-}\mathsf{AC}$, is the axiom scheme consisting of sentences of the form
\[\forall \bar Z \, \forall \bar z\, [\forall x\, \exists Y\, \varphi(x,Y,\bar Z, \bar z) \Rightarrow \exists \tuple{Y_x}_{x \in \bbn}\, \forall x\, \varphi(x, Y_x, \bar Z, \bar z)],\]
where $\varphi$ is $\Sigma^1_2$.
\end{definition}
\begin{theorem}\label{thm:choice}(see \cite[Theorem VII.6.9]{simpson})
$\Pi^1_2\ca \vdash \Sigma^1_2\textrm{-}\mathsf{AC}$.
\end{theorem}

A sequence $\tuple{X_i}_{i \in \bbn}$ can be regarded as a countable model for the language $L_2$, with $\bbn$ as the first-order universe and $\{X_i\}_{i \in \bbn}$ as the second-order universe. Such a model is called \emph{a countable coded model} and typically denoted $M$.
\begin{definition}
A countable coded model $M$ is a \emph{$\beta_2$-model} if for all $\bar x \in \bbn,\bar  X \in M,$ and each $\Pi^1_2$ formula $\varphi(\bar x,\bar X)$,
\[M \models \varphi(\bar x, \bar X) 
\textrm{ iff } \varphi(\bar x,\bar X) \textrm{ is true}.\]
\end{definition}
\noindent (A completely formal version of this definition involves truth definitions for $\Pi^1_2$ sentences and can be found for instance in \cite[Definition VII.7.2]{simpson}). Note that any true $\Pi^1_3$ {sentence} remains true in each $\beta_2$ model.

\begin{theorem}(see \cite[Theorem VII.6.9 and VII.7.4]{simpson})\label{thm:beta}
$\Pi^1_2\ca$ proves:
\[\forall X\, \exists M\, (M \textrm{ is a }\beta_2\textrm{-model and  } X\in M).\]
\end{theorem}

The final principle we have to discuss concerns fixpoints of iterations of certain monotone operators. A \emph{$\Delta^1_2$ monotone operator} is given by a pair of $\Pi^1_2$ formulas $\varphi(x,X), \psi(x,X)$ (possibly with parameters) such that \[\forall x\, \forall X\, (\varphi(x,X) \Leftrightarrow \neg\psi(x,X))\] and \[\forall X_1\, \forall X_2 \, \forall x\, (X_1 \subseteq X_2 \land \varphi(x,X_1)) \Rightarrow \varphi(x,X_2).\] We think of $\varphi$ and $\psi$ as defining an operator $\Gamma_{\varphi,\psi}$ on sets that maps a set $X$ to $\{x \in \bbn: \varphi(x,X)\}$.

\begin{definition}\label{def:fixpoint}
The axiom scheme $\Delta^1_2\textrm{-}\mathsf{MI}$ asserts the following for every $\Delta^1_2$ monotone operator given by formulas $\varphi,\psi$: there exists a prewellordering (reflexive transitive relation connected on its field) $\preccurlyeq$ with field $P$ such that \[\forall x\, [x \in P \Leftrightarrow x \in \Gamma_{\varphi,\psi}( \{y: y \prec x\})] \textrm{ and } \Gamma_{\varphi,\psi}(P) \subseteq P.\] 
\end{definition}
\noindent If $\preccurlyeq, P$ are as in Definition \ref{def:fixpoint}, then $P$ can be regarded as the inductively generated least fixed point of 
$\Gamma_{\varphi,\psi}$.

\begin{theorem}\cite[Theorem 5.1]{medsalem_tanaka_delta_03}\label{thm:fixpoint}
$\Pi^1_2\ca \vdash \Delta^1_2\textrm{-}\mathsf{MI}$.
\end{theorem}

\paragraph{Notational convention.} As above, we will use the letter $\bbn$ to denote the natural numbers as formalized in second-order arithmetic, that is, the domain of the first-order sort. On the other hand, the symbol $\omega$ will stand for the concrete, or standard, natural numbers. For instance, given a theory $\mathrm{T}$ and a formula $\varphi(x)$, ``$\mathrm{T}$ proves $\varphi(n)$ for all $n \in \omega$'' will mean ``$\mathrm{T} \vdash \varphi(0), \mathrm{T} \vdash \varphi(1), \ldots$'', which does not have to imply $\mathrm{T} \vdash \forall x\!\in\!\bbn\, \varphi(x)$''.

The letter $n$ will be used exclusively to denote elements of $\omega$. As formal variables of the first-order sort, we will typically use $x, y, z,\ldots$, but sometimes also 
$i,j,k, \ell$ or other lowercase letters different from $n$.


\subsection{Automata and MSO}

The study of monadic second order logic on the infinite binary tree relies heavily on the theory of automata on infinite words and infinite trees. Our presentation of automata theory and MSO is based on \cite{thomas}. 

Given a finite set $\Sigma$ (the so-called \emph{alphabet}), a \emph{word over alphabet $\Sigma$} is simply a mapping $f \colon \bbn \to \Sigma$. A \emph{tree} (or more precisely, \emph{labelled binary tree}) \emph{over $\Sigma$} is a mapping $T \colon \{0,1\}^* \to \Sigma$, where $\{0,1\}^*$ stands for the set of all finite binary strings\footnote{In automata theory, a more typical choice of symbols would be $w$ and $t$ instead of, respectively, $f$ and $T$. However, in this paper we decided to use capital letters to emphasize that these are objects of the second-order sort.}. Note that words and trees over $\Sigma$ have another natural representation as structures with universe $\bbn$ resp.\ $\{0,1\}^*$ and $\mathrm{card}(\Sigma)$ disjoint unary relations whose union is the universe.

\begin{definition}[Nondeterministic Parity Word Automaton]
A \emph{nondeterministic parity word automaton} over a finite alphabet $\Sigma$ is a tuple $\mathcal{A}\!=\!\langle \Sigma,Q, q^I, \Delta, \mathrm{rk}\rangle$ where $Q$ is a finite \emph{set of states}, $q^I\!\in\!Q$ is the \emph{initial state}, $\Delta\subseteq Q\times \Sigma \times Q$ is the \emph{transition relation}, and $\mathrm{rk}\colon Q\to\omega$ is the \emph{rank function}.
\end{definition}
We use the letter $\delta$ to denote individual \emph{transitions} of $\Aa$, i.e.\ elements of $\Delta$.

A \emph{run} of a nondeterministic parity word automaton $\Aa$ on a word $f$ over $\Sigma$ is a labelling $\rho\colon\bbn\to\Delta$ which is  consistent, that is for every position $i\in\bbn$ in $f$, if $\rho(i) = (q_i,a_i,q_{i+1})$, then
\begin{enumerate}
\item $f(i) = a_i$,
\item $\rho(i+1) =  (q_{i+1},a_{i+1},q_{i+2})$,
\end{enumerate}
and $q_0=q^I$, that is the first state in the run is the initial state of the automaton. Intuitively, the transitions in a run have to be such that after moving from a position $i$ to the successor position $i+1$, the run continues from the state $q_{i+1}$ indicated by the transition $\rho(i)$.

A run $\rho$ on a word $f$ is {\em accepting} if the states $q_0,q_1,\ldots$  along this run satisfy \emph{the parity acceptance condition}: \[\liminf_{i \in \bbn} \mathrm{rk}(q_i) \textrm{ is even.}\] 

A word $f$ is {\em accepted} by the automaton $\Aa$ if there exists an accepting run of ${\mathcal A}$ on $f$. 

\begin{definition}[Nondeterministic Parity Tree Automaton]
A \emph{nondeterministic parity tree automaton} over a finite alphabet $\Sigma$ is a tuple $\mathcal{A}\!=\!\langle \Sigma,Q, q^I, \Delta, \mathrm{rk}\rangle$ where $Q$ is a finite \emph{set of states}, $q^I\!\in\!Q$ is the \emph{initial state}, $\Delta\subseteq Q\times \Sigma \times Q\times Q$ is the \emph{transition relation}, and $\mathrm{rk}\colon Q\to\bbn$ is the \emph{rank function}.
\end{definition}
If a transition $\Delta \ni \delta$ has the form $(q,a,q_0,q_1)$, we may write $q_{0} = \delta(a,q,0)$ and $q_{1} = \delta(a,q,1)$. The idea is that if the automaton is in some vertex of the tree and reads the letter $a$ while in state $q$, it may use $\delta$ to move simultaneously to the left son of the vertex in state $q_0$ and to the right son in state $q_1$.

A \emph{run} of a nondeterministic parity tree automaton $\Aa$ on a tree $T$ over $\Sigma$ is a labelling $\rho \colon \{0,1\}^* \to \Delta$ which is consistent, that is for every vertex $v\in \{0,1\}^*$ if $\rho(v) = (q_v,a_v,q_{v0},q_{v1})$, then 
\begin{enumerate}
\item $t(v) = a_v$,
\item $\rho(vi) =  (q_{vi},a_{vi},q_{vi0},q_{vi1})$ for $i=0,1$,
\end{enumerate}
and $q_\epsilon = q^I$ where $\epsilon$ denotes the empty string. 
Intuitively, the transitions in a run have to be chosen so that after moving from a vertex $v$ in the tree to its sons $v0, v1$, the run continues from the states $q_{v0}, q_{v1}$ indicated by the transition $\rho(v)$.

A run $\rho$ on a tree $T$ is {\em accepting} if for every branch $\pi \in \{0,1\}^\bbn$, the states along this branch,  $q_{\pi{\upharpoonright_0}}, q_{\pi{\upharpoonright_1}}, \ldots$  satisfy \emph{the parity acceptance condition}: 
\[\liminf_{i \in \bbn} \mathrm{rk}(q_{\pi{\upharpoonright_i}}) \textrm{ is even.}\]

A tree $T$ is {\em accepted} by the automaton $\Aa$ if there exists an accepting run of ${\mathcal A}$ on $T$. 

\begin{definition}[Deterministic automaton] We call a nondeterministic parity word automaton $\mathcal{A}\!=\!\langle \Sigma,Q, q^I, \Delta, \mathrm{rk}\rangle$ a {\em deterministic word automaton} if the transition relation $\Delta$ is a graph of a function $Q\times \Sigma\to Q$. Similarly, we call a nondeterministic parity tree automaton $\mathcal{A}\!=\!\langle \Sigma,Q, q^I, \Delta, \mathrm{rk}\rangle$ a {\em deterministic tree automaton} if the transition relation $\delta$ is a graph of a function $Q\times \Sigma\to Q\times Q$. 
\end{definition}
Note that a deterministic word (resp.\ tree) automaton has exactly one run on each word (resp.\ tree).

\medskip

We consider monadic second order logic MSO over the structure $(\{0,1\}^*,S_0,S_1)$, where $S_0$ and $S_1$ are the left and right two successor relations, respectively ($S_0(v,w)$ holds iff $w = v0$ and $S_1(v,w)$ holds iff $w = v1$). The language of MSO over $(\{0,1\}^*,S_0,S_1)$ contains first-order variables $x,y,\ldots$ ranging over elements of $\{0,1\}^*$ and second-order variables $X,Y,\ldots$ ranging over subsets of $\{0,1\}^*$. Atomic formulas have the form $x=y$, $S_0(x,y)$, $S_1(x,y)$ and $x \in X$. The language of MSO has the usual boolean connectives and the quantifiers $\exists x, \exists X$. In the language with just the two successors, as opposed to the arithmetical language with $+$ and $\cdot$, there is no way to define a pairing function in MSO, so the restriction to \emph{unary} second-order quantifiers is very important.

\subsection{Games}

We will be concerned with games in which winning conditions take the form of boolean combinations of $\Sigma^0_2$ properties. To talk about arbitrary boolean combinations of $\Sigma^0_2$ statements, we formalize a version of the so-called \emph{difference hierarchy} over $\Sigma^0_2$.

Let $f$ be a distinguished second-order variable which is assumed to represent (the graph of) a function from $\bbn$ to $\bbn$. 
A \emph{$(\Sigma^0_2)_x$ formula} $\varphi(f)$ is given by a number $x$ and a $\Pi^0_2$ formula $\psi(y,f)$, possibly with other parameters, such that for all $f$, $\psi(x,f)$ always holds, and for all $z < y < x$, if $\psi(z,f)$, then $\psi(y,f)$. We say that \emph{$\varphi(f)$ holds} if
\[\psi(0,f) \vee \psi(2,f) \vee \ldots \vee \psi(2\lfloor x/2 \rfloor, f);\]
or formally, if the smallest $y \le x$ such that $y = x \vee \psi(y,f)$ is even. 

Note that in $\aca$ a $(\Sigma^0_2)_1$ formula is the same thing as a $\Pi^0_2$ formula, in the sense that for every $(\Sigma^0_2)_1$ formula
$\varphi(f)$ there is a $\Pi^0_2$ formula $\xi(f)$ such that
$\aca \vdash \forall f\,(\varphi(f) \Leftrightarrow \xi(f))$, and vice versa. Similarly, a $(\Sigma^0_2)_{x+1}$
formula is the same thing as the disjunction of a $\Pi^0_2$ formula and a negated $(\Sigma^0_2)_x$ formula. So, our class $(\Sigma^0_2)_x$ is dual to the usual $x$-th level of the difference hierarchy over $\Sigma^0_2$. It is a matter of routine if tedious verification that every concrete boolean combination of $\Sigma^0_2$ properties can be expressed by a $(\Sigma^0_2)_n$ formula for some $n \in \omega$.

A \emph{$(\mathbf{\Sigma^0_2})_x$ Gale-Stewart game} (briefly,
a \emph{$(\mathbf{\Sigma^0_2})_x$ game}) is given by a $(\Sigma^0_2)_x$ formula $\varphi(f)$, again, possibly with other parameters (in accordance with the conventions from descriptive set theory, the boldface font serves precisely to indicate the possible presence of parameters). The game is played by two players, 0 and 1, who alternately choose natural numbers $f(0), f(1), \ldots $, building an infinite sequence $f \in \bbn^\bbn$. Player 0 wins the game if $\varphi(f)$ holds, and player 1 wins otherwise. The notions of a \emph{strategy} and \emph{winning strategy}  for each player are defined as usual. The game given by $\varphi(f)$ is \emph{determined} if one of the players has a winning strategy. For precise definitions, see e.g.\ \cite{kechris}.

\begin{definition}
$\sigmadet{x}$ is the $\Pi^1_3$ sentence ``all $(\mathbf{\Sigma^0_2})_x$ games are determined''.

$\sigmadet{x}^*$ is the same statement restricted to games in the Cantor space $\{0,1\}^\bbn$ instead of the Baire space $\bbn^\bbn$, that is to games where the players are required to choose only numbers from $\{0,1\}$ in each move. 
\end{definition}

Obviously, $\sigmadet{x}$ implies $\sigmadet{x}^*$. On the other hand, it easily follows from \cite[Lemma 4.2]{nemoto_medsalem_tanaka} that, provably in $\aca$, $\sigmadet{x+1}^*$ implies $\sigmadet{x}$. So, we have:
\begin{proposition}\label{prop:bairecantor}
$\ouraxiom$ and $\ouraxiomcantor$ are equivalent in $\aca$. 
\end{proposition}

A \emph{parity game of index $(0,x)$} is a tuple $G = \tuple{V_\exists, V_\forall, v_I, E, \mathrm{rk}}$, where: $V_\exists$ and $V_\forall$ are disjoint sets with union $V:= V_\exists \cup V_\forall$; $v_I$ is an element of $V$; $E \subseteq V^2$ is such that for every $v \in V$ there is some $w$ with $(v,w) \in E$; and the \emph{rank function} $\mathrm{rk}$ is a function from $V$ to 
$\{0,\ldots,x\}$. The game is played on the \emph{arena} $V$ by the two players $\exists$ and $\forall$. The game starts in the \emph{initial position} $v_I$; and if it reaches position $v \in V_p$, $p \in \{\exists,\forall\}$, then player $p$ moves to some $w$ such that $(v,w) \in E$. Formally, a \emph{play} of $G$ is a sequence $\tuple{v_i}_{i \in \bbn}$ such that $v_0 = v_I$ and $(v_i, v_{i+1}) \in E$ for all $i$. Player $\exists$ wins the play exactly if
\[\liminf_{i \in \bbn}\mathrm{rk}(v_i) \textrm{ is even}.\]

$\sigmadet{x}$ implies that all parity games of index $(0,x)$ are determined. However, we also need a stronger notion of determinacy. A \emph{positional strategy} (also known as a \emph{memoryless}, or \emph{forgetful}, strategy) for $\exists$ (resp. $\forall$) is a function $\sigma$ from $V_\exists$ (resp.\ $V_\forall$) into $V$ such that for all $v$ in the domain, $(v,\sigma(v))\in E$. A positional strategy $\sigma$ is \emph{winning} if the player using $\sigma$ wins every play consistent with $\sigma$. The game $G$ is \emph{positionally determined} if one of the two players has a positional winning strategy. Basic notions related to parity games come originally from \cite{emerson_jutla}.

\begin{definition}[Tree-like parity games]
A parity game $G$ is \emph{tree-like} if $V = \{0,1\}^{<\bbn}$, $V_\exists = \bigcup_{i \in \bbn} \bbn^{2i}$, and $(v,w)\in E$ iff
$w = v0$ or $w = v1$. 
\end{definition}
Our notion of $(\Sigma^0_2)_x$ formula was chosen to make the proof of the following lemma rather straightforward.

\begin{lemma}\label{lem:dwakropkacztery}
$\aca$ proves that for every $x$, $\sigmadet{x}^*$ holds  if and only if all tree-like parity games of index $(0,x)$ are positionally determined.
\end{lemma} 
\begin{proof}
Clearly, a tree-like parity game is determined iff it is positionally determined. Moreover, such a game of index $(0,x)$ can be thought of as a $(\mathbf{\Sigma^0_2})_x$ Gale-Stewart game in the Cantor space given by $x$ and the formula 
\[\psi(y,f):= \eb{z}{y}\forall i \, \exists j\!>\!i\,(\mathrm{rk}(f{\upharpoonright_j}) = z).\]
Note that $\psi$ is $\Pi^0_2$ in $\aca$. This proves the left-to-right direction.

In the other direction, consider a $(\mathbf{\Sigma^0_2})_x$ Gale-Stewart game given by the $\Pi^0_2$ formula $\psi(y,f)$. By a normal form result for $\Pi^0_2$ analogous to the normal form for $\Sigma^0_1$ presented as \cite[Theorem II.2.7]{simpson}, we may assume that $\psi(y,f)$ has the shape $\forall z\, \exists u\, \delta(y,z,u,f{\upharpoonright_u})$ with all quantifiers in $\delta$ bounded by $u$. Define the tree-like parity game of index $(0,x)$ by letting $\mathrm{rk}(f{\upharpoonright_i})$ equal the smallest such $y <  x$ that for some $j$, 
\[\as{z}{j}\eb{u}{i}\delta(y,z,u,f{\upharpoonright_u})\]
holds, but the number of $k < i$ with $\mathrm{rk}(f{\upharpoonright_i}) = y$ is strictly smaller than $j$. If there is no such $y$, let $\mathrm{rk}(f{\upharpoonright_i})$ equal $x$. It is easy to verify in $\aca$ that for every $f \in \{0,1\}^\bbn$, $\liminf_i\mathrm{rk}(f{\upharpoonright_i})$ is exactly the smallest $y$ such that
$y = x \vee \psi(y,f)$. Thus, determinacy of the parity game implies  determinacy of the game given by $x$ and $\psi$. 
\end{proof}

\section{Complementation}\label{sec:complementation}

In this section, we study the logical strength of the complementation theorem for tree automata. We prove that the theorem is equivalent to a determinacy principle intermediate between $\ouraxiom$ and the positional determinacy of all parity games. The somewhat technical definition of the class of games appearing in the principle is abstracted from the so-called Automaton-Pathfinder games appearing quite naturally in the proofs of Rabin's theorem (see e.g.~\cite{thomas}).

The proof of our characterization works in $\aca$ and, unlike the proofs in later sections of this paper, does not rely on any earlier results related to the logical strength of determinacy principles.

\begin{definition}
A parity game $G$ is \emph{almost tree-like} if $V = \{0,1\}^{*} \times \{0,\ldots,k\}$ for some $k \in \bbn$, $V_\exists = \bigcup_{i \in \bbn} \bbn^{2i} \times \{0,\ldots,k\}$, and $((v,i),(w,j))\in E$ implies $(w = v0 \lor w = v1)$. 
\end{definition}

\begin{theorem}\label{thm:complementation} 
The following are equivalent over $\aca$: 
\begin{itemize}
\item[$(1)$] all almost tree-like parity games are positionally determined,
\item[$(2)$] for every tree automaton $\Aa$ there exists a tree automaton $\Bb$ such that for any tree $T$, $\Bb$ accepts $T$ iff $\Aa$ does not accept $T$.
\end{itemize}
\end{theorem}

\noindent
From Theorem \ref{thm:complementation} and Lemma \ref{lem:dwakropkacztery} we immediately get:

\begin{corollary}\label{cor:complementation} 
$\aca$ proves the implications $(1) \Rightarrow (2)$ and $(2) \Rightarrow (3)$ for:
\begin{itemize}
\item[$(1)$] all parity games are positionally determined,
\item[$(2)$] for every tree automaton $\Aa$ there exists a tree automaton $\Bb$ such that for any tree $T$, $\Bb$ accepts $T$ iff $\Aa$ does not accept $T$,
\item[$(3)$] $\ouraxiom$, or equivalently, all tree-like parity games are positionally determined.
\end{itemize}
\end{corollary}

The remainder of this section is devoted to the proof of Theorem \ref{thm:complementation}.

\begin{proof}[Proof of Theorem \ref{thm:complementation}]
\noindent $(1) \Rightarrow (2)$. We formalize a standard proof of the complementation theorem for tree automata, similar e.g.\ to the one in \cite{thomas}. Let $\Aa = \{ \Sigma_\Aa, Q_\Aa, q^I_\Aa, \Delta_\Aa, \mathrm{rk}_\Aa\}$ be a parity tree automaton. We may assume w.l.o.g.\ that for each $a \in \Sigma_\Aa$, $q \in Q_\Aa$, there is at least one transition in $\Delta_\Aa$ of the form $(q,a,\cdot,\cdot)$. This is because $\Aa$ can be easily modified so as to satisfy this condition without changing the class of accepted trees.

Given a labelled binary tree $T$, consider the following game $G_{\Aa,T}$ between two players, Automaton and Pathfinder. The set of Automaton's positions is $\{0,1\}^{*} \times Q_\Aa$, with $(\emptyset, q^I_\Aa)$ the starting position, while the set of Pathfinder's positions is $\{0,1\}^{*} \times \Delta_\Aa$. Given a position $(w,q)$, Automaton can choose a transition from $q$, that is, move to the position $(w,\delta)$, where 
$\Delta_\Aa \ni \delta = (q,T(w),q_0,q_1)$. Pathfinder responds by deciding which direction to take from $w$, that is, by moving either to position $(w0, q_0)$ or to $(w1, q_1)$. A play of the game induces a sequence of states $\tuple{q_i}_{i \in \bbn} \in (Q_\Aa)^\bbn$, and Automaton wins if and only if
\[\liminf_{i \in \bbn}\mathrm{rk}_\Aa(q_i) \textrm{ is even}.\]
Note that because of the assumption that $\Aa$ has at least one transition from every letter and state, $G_{\Aa,T}$ can easily be formalized as an almost tree-like parity game.

$G_{\Aa,T}$ is defined in such a way that $\Aa$ accepts $T$ if and only if Automaton has a positional winning strategy. Thus, by positional determinacy, $\Aa$ does not accept $T$ if and only if Pathfinder has a positional winning strategy in $G_{\Aa,T}$.

Let $\Ss$ be the (finite) set of all maps from $\Delta_\Aa$ into $\{0,1\}$. Note that a positional strategy for Pathfinder can be represented as a labelled binary tree  $S$ such that $S(w) \in \Ss$
for all $w \in \{0,1\}^{*}$. The strategy is winning if for any choice of transitions $\tuple{\delta_w}_{w \in \{0,1\}^{*}}$ consistent with the labelling of $T$, the ranks of states of $\Aa$ appearing on the path determined by $S$ and 
$\tuple{\delta_w}_{w \in \{0,1\}^{*}}$ do not satisfy the parity condition.

We will construct an automaton $\Bb$ which accepts a tree $T$ if and only Pathfinder has a positional winning strategy in $G_{\Aa,T}$. The construction of $\Bb$ proceeds in four standard steps, of which three are elementary and one invokes McNaughton's determinization theorem for word automata (\cite{mcnaughton}), discussed separately below.

{In the first step} of the construction
we build a deterministic word automaton 
$\Aa_1$ which accepts all infinite words 
$\tuple{(s_i,a_i,\delta_i,\pi_i)}_{i \in \bbn} 
\in (\Ss \times \Sigma_\Aa \times \Delta_\Aa\times \{0,1\})^\bbn$  
such that if  
\begin{equation}\tag{$\star$} \forall i \! \in \! \bbn\, (s_i(\delta_i) = \pi_i)
\label{rule1}
\end{equation}
and if we define
\begin{align*}
q_0 & = q^I_\Aa \\
q_{i+1} & = \delta_i(a_i,q_i,\pi_i),
\end{align*}
then either at some point $q_{i+1}$ cannot be defined (i.e.\ $\delta_i$ is not a transition from state $q_i$ and letter $a_i$) or
\begin{equation*} 
\liminf_{i \in \bbn} \mathrm{rk}_\Aa(q_i)\ \text{is odd}.
\label{rule2}
\end{equation*}
The set of states of $\Aa_1$ is $Q_\Aa \cup \{\top\}$, 
where $\mathrm{rk}_{\Aa_1}(q) = \mathrm{rk}_{\Aa}(q) + 1$ for $q \in Q_{\Aa_1}$ 
and $\top$ is an additional accepting state (i.e.\ with rank 0). The transitions are defined in the natural way, except that whenever the rule (\ref{rule1}) is violated or $q_{i+1}$ cannot be defined, we redirect the computation to state $\top$.

{In the second step} we consider the following property of
infinite words $\tuple{(s_i,a_i,\pi_i)}_{i \in \bbn} \in (\Ss \times \Sigma_\Aa \times \{0,1\})^\bbn$: for every sequence 
$\tuple{\delta_i}_{i \in \bbn} \in (\Delta_\Aa)^\bbn$,
the word $\tuple{(s_i,a_i,\delta_i,\pi_i)}_{i \in \bbn}$ is
accepted by $\Aa_1$. Note that thanks to the fact that $\Aa_1$ is deterministic, the complement of this property is recognized by a nondeterministic word automaton. 
By determinization of word automata (\cite{mcnaughton}, see below) we can find a deterministic parity word automaton $\Aa_2$ recognizing this property. 


{In the third step} we define a tree automaton $\Aa_3$ over the alphabet $\Ss\times\Aa$ such that a tree $(S,T)$ is accepted if for every infinite path $\pi\in \{0,1\}^\bbn$, if $\tuple{(s_i,a_i)}_{i\in \bbn}$ is the sequence of labels appearing on this path, then the infinite word $\tuple{(s_i,a_i,\pi_i)}_{i \in \bbn}$ is accepted by the automaton $\Aa_2$. So, $\Aa_3$ accepts $(S,T)$ iff $S$ encodes a positional winning strategy for Pathfinder in $G_{\Aa,T}$. The states of $\Aa_3$ are the same as those of $\Aa_2$, and defining the transition function is unproblematic thanks to the fact that $\Aa_2$ is deterministic.


Finally, {in the fourth step} we define the nondeterministic tree automaton $\Bb$ as accepting a given tree $T$ over $\Sigma_\Aa$ if there exists a labelling 
$S \colon \{0,1\}^{*} \to \Ss$ such that the tree
$(S,T)$ is accepted by $\Aa_3$. During its computation on $T$, $\Bb$ uses nondeterminism to ``guess'' the labels $S(w)$ for 
$w \in \{0, 1\}^*$. By construction $\Bb$ accepts $T$ iff Pathfinder has a positional winning strategy in $G_{\Aa,T}$.

\emph{Determinization of word automata.} To complete the proof of $(1) \Rightarrow (2)$, we need to make sure that $\aca + (1)$ is able to prove that McNaughton's result \cite{mcnaughton} that every nondeterministic word automaton is equivalent to a deterministic automaton. In fact, this is provable in $\aca$ and does not require the full power of $\aca$.

The determinization of a nondeterministic word automaton with a parity condition proceeds in two steps. The first is to replace the original automaton by a nondeterministic B\"uchi automaton (i.e.\ with index $(0,1)$). This is straightforward: assume that the original automaton, say $\Aa$, has index $(0,x)$, so that $\Aa$ accepts a word if it has a computation in which the $\liminf$ of ranks of states is some odd number $y \le x$. The B\"uchi automaton $\Bb$ behaves like $\Aa$, except that it additionally uses nondeterminism to do two things. Firstly, in the very first transition it guesses the value of $y$. Secondly, at some point in the computation it guesses that states with rank $> y$ will no longer appear; from that point onwards, it assigns rank $0$ to states with rank $y$ in $\Aa$, rank $1$ to states with rank $< y$ in $\Aa$, and aborts the computation if $\Aa$ wants to enter a state with rank $> y$. 

It remains is to build a deterministic parity word automation simulating a nondeterministic B\"uchi word automaton $\Bb$. This can be carried out by means of the  \emph{Safra construction}, originally due to \cite{safra1}. This is a sophisticated refinement of the classical power set construction used for finite word automata (on finite words, a nondeterministic automaton with set of states $Q$ can be simulated by a deterministic automation with set of states $\Pp(Q)$). In the Safra construction, the states of the new deterministic automaton $\Dd$ are certain (bounded-size) trees with vertices labelled by letters from a certain fixed finite alphabet. 
The combinatorial details of the construction are a little involved, but the logical strength engaged is quite modest. The construction of $\Dd$ from $\Bb$ is completely elementary, the proof that acceptance of $\Bb$ implies acceptance of $\Dd$ requires nothing beyond defining number sequences by recursion with an arithmetical condition in the recursion step, and the other direction additionally makes use of K\"onig's Lemma in the form known as \emph{weak} K\"onig's Lemma, $\mathsf{WKL}$; all this is well-known to be readily formalizable in $\aca$. (For details of the Safra construction, see e.g.\ \cite{piterman,thomas}\footnote{The proof presented in \cite{thomas} yields a deterministic automaton with the so-called Rabin acceptance condition, which can be simulated without difficulty by a parity condition; the modified construction of \cite{piterman} leads directly to a deterministic parity automaton.}.)

\begin{remark}
We have not attempted a careful verification, but we believe that the proof of determinization for word automata goes through in the fragment of $\aca$ known as $\wkl$ extended by the induction scheme for $\Sigma^0_2$ formulas. Without $\Sigma^0_2$ induction, the basic notions of automata theory on infinite structures make little sense, in particular the $\liminf$ of ranks appearing in a computation might not exist.  
\end{remark}

\medskip

\noindent $(2) \Rightarrow (1)$. 
For any fixed $x,k \in \bbn$, we can represent almost tree-like parity games with index $(0,x)$ and arena $\{0,1\}^{*}\times\{0,\ldots,k\}$ by labelled binary trees over a suitable alphabet $\Sigma_{x,k}$. The label of a node $v$ in such a tree lists the ranks of all positions $(v,i)$ and lists all pairs $((v,i),(v0,j))$ and 
$((v,i),(v1,j))$ in $E$. Thus, $\Sigma_{x,k}$ should contain $(x+1)^{k+1} \cdot 2^{2(k+1)^2}$ symbols.

Assume $\neg(1)$ and let $x \in \bbn$ be such that there is an undetermined almost tree-like parity game with index $(0,x)$, set of positions $\{0,1\}^{*}\times\{0,\ldots,k\}$ and edge relation $E$ such that $((v,i),(w,j)) \in E$ implies $w = vb$ for some $b \in \{0,1\}$.

A tree over $\Sigma_{x,k}$ representing a game $G$ has the property 
$W^\exists_{0,x,k}$ (resp.\ $W^\forall_{0,x,k}$) if $\exists$ (resp. $\forall$) has a positional winning strategy in $G$. Note that both
$W^\exists_{0,x,k}$ and $W^\forall_{0,x,k}$ can be expressed by MSO sentences with 4 blocks of quantifiers. For instance, the sentence
expressing $W^\exists_{0,x,k}$ is the prenex normal form of
\[\exists S_1\, \ldots \,\exists S_\ell\, 
\forall P_1\, \ldots\, \forall P_r\, \left[\alpha(\bar S, \bar P) \vee \bigvee_{0 \le y \le \lfloor x/2 \rfloor} \beta_{2y}(\bar S, \bar P)\right],\]
where the $S$'s represent a positional strategy for player $\exists$, the $P$'s represent potential plays of the game, $\alpha$ is a purely existential first-order formula stating that the play $\bar P$ is inconsistent with the strategy $\bar S$, and each
$\beta_{2y}$ is an $\forall\exists \vee \exists \forall$ first-order formula stating that $2y$ is the $\liminf$ of ranks appearing in the play $\bar P$. The sentence expressing $W^\forall_{0,x,k}$ is defined analogously.

Now assume $(2)$. It is routine to verify in $\aca$ that any quantifier-free expressible property of labelled trees is recognized by a nondeterministic tree automaton. The usual argument by induction on formula complexity  (see e.g.\ \cite[Theorem 6.7, cf.\ Theorem 3.1]{thomas}), using (2) in the step for $\neg$ and the nondeterminism of automata in the step for a block of second-order $\exists$'s, proves that every property of labelled trees expressible by an MSO sentence with at most $5$ quantifier blocks can be recognized by a tree automaton. Hence, for any fixed $x,k$ there is an automaton $\Aa_{x,k}$ recognizing $\neg W^\exists_{0,x,k} \wedge  \neg W^\forall_{0,x,k}$. It remains to show that this implies (1).

A labelled tree $T$ is \emph{regular} if there is some
bound $d \in \bbn$ with the following property: for every vertex $v \in \{0,1\}^{*}$ there exists $w \in \{0,1\}^{< d}$ such that the subtree of $T$ under $w$ is the same as the subtree under $v$, that is, for all $u \in \{0,1\}^{*}$,
$T(vu) = T(wu)$. We complete the argument by proving two lemmas about regular trees which together immediately imply that $\Aa_{x,k}$ cannot accept any tree, and therefore (1) holds. The first lemma expresses a completely standard fact, whereas the second is interesting only in a context where positional determinacy of parity games is not known in advance.

\begin{lemma}\label{lem:regular-tree}
$\aca$ proves that any tree automaton which accepts some tree also accepts a regular tree.
\end{lemma}
\begin{proof}
Let $\Aa$ be a tree automaton of index $(0,x)$. Consider the following parity game $G_\Aa$. The game arena is split into 
$V_\exists = Q_\Aa$ and $V_\forall = \Delta_\Aa$ with $q^I_\Aa$ as the initial position. When the game is in position $q \in Q_\Aa$, player $\exists$ moves to some $\delta \in \Delta_\Aa$ of the form $(q,a,q',q'')$ for some $a \in \Sigma_\Aa, q', q'' \in Q_\Aa$. Now the game is in position $\delta$, and player $\forall$ decides whether to go left or right, that is, moves to either $q'$ or $q''$. Player $\exists$ wins if $\liminf$ of ranks of states visited during the play is even (formally this is achieved by setting the ranks of all $\delta \in \Delta_\Aa$ to $x$).

The positional determinacy of parity games on finite arenas is known to be provable even in $S^2_2$, a very weak subtheory of $\aca$ \cite{beckmann-moller:parity}. Moreover, if $\Aa$ accepts some tree $T$, then player $\exists$ clearly has a winning strategy in $\Aa$, and thus also a positional winning strategy, say $\sigma$. Now construct a labelled tree $T_\sigma$ in the following way: to determine the label $T(w)$ for $w \in \{0,1\}^{*}$, simulate a play of $G_\Aa$ in which $\exists$ plays according to $\sigma$, $\forall$ chooses directions so as to reach $w$, and for each $i \le \mathrm{lh}(w)$ the label $T(w{\upharpoonright_i})$ is chosen to correspond to the transition $\delta$ chosen by $\exists$. $T_\sigma$ is accepted by $\Aa$ because $\sigma$ is a winning strategy. Additionally, $T_\sigma$ is a regular tree because $G_\Aa$ has a finite arena and $\sigma$ is positional.
\end{proof}

\begin{lemma}\label{lem:regular-determined}
$\aca$ proves that for each $x$ and $k$, the almost treelike parity game represented by a regular tree $T$ over $\Sigma_{x,k}$ is determined.
\end{lemma}

\begin{proof}[Proof of Lemma \ref{lem:regular-determined}]
Let $T$ be a regular tree over $\Sigma_{x,k}$ encoding an almost treelike parity game $G$. Define $U$ to be the set consisting of those $v\in \{0,1\}^{*}$ for which there is no $w$ length-lexicographically smaller than $v$ such that the lengths of $v$ and $w$ are congruent mod $2$ and $T(vu) = T(wu)$ for all $u$ (the additional congruence condition is needed to distinguish between moves of different players). 
The formula defining $U$ is arithmetical, so $U$ is indeed a set. Moreover, since $T$ is regular, the set $U$ is finite. 

Consider the following parity game $H$ with arena 
$V := U \times \{0,\ldots, k\}$. The assignment of positions to players and the ranks of positions are inherited from $G$. A move from $(v,i)$ to $(w,j)$ is possible in $H$ if and only if the subtree of $T$ under $w$ is identical to the subtree under a son $vb$ of $v$ and the label $T(v)$ indicates that a move from $(v,i)$ to $(vb, j)$ is possible in $G$. The game $H$ exists by arithmetical comprehension. 
By positional determinacy of parity games on finite arenas, $H$ 
is positionally determined provably in $\aca$. Let $\sigma_H$ be a positional winning strategy for one of the players, say $\exists$.

We can use $\sigma_H$ to define a positional strategy $\sigma_G$ in the game $G$ represented by $T$: at position $(v,i)$, player $\exists$ finds the unique vertex $w \in U$ such that the lengths of $w$ and $v$ are congruent mod $2$ and $T$ under $v$ is the same as $T$ under $w$, and then moves from $(v,i)$ analogously to the move $\sigma_H$ would make at $(w,i)$. The strategy $\sigma_G$ exists by arithmetical comprehension, and since any play of $G$ consistent with $\sigma_G$ corresponds to a play of $H$ consistent with $\sigma_H$, $\sigma_G$ is in fact a winning strategy for $\exists$ in $G$.
\end{proof}

\noindent
The proofs of the Lemmas conclude the proof of Theorem \ref{thm:complementation}.
\end{proof}

\section{Determinacy vs comprehension}\label{sec:determinacy-comprehension}

We now turn our attention to the question which set-theoretic existence axioms are needed to prove the determinacy statements
considered in Theorem \ref{thm:complementation} and Corollary
\ref{cor:complementation}. For the statement $\ouraxiom$, 
sharp upper and lower bounds on the requisite amount of comprehension follow from earlier work on determinacy statements in second-order arithmetic, most importantly \cite{mollerfeld_thesis, heinatsch_mollerfeld, medsalem_tanaka_delta_03}.

We begin with the upper bound. The following result is a direct corollary of the proof of Theorem 4.3 in \cite{medsalem_tanaka_delta_03}:

\begin{theorem}[\cite{medsalem_tanaka_delta_03}]\label{thm:det_x_upperbound}
$\aca + \Delta^1_2\textrm{-}\mathsf{MI} + \Sigma^1_3\ind 
\vdash \ouraxiom$.
\end{theorem}
By Theorem \ref{thm:fixpoint}, it follows that $\ouraxiom$
is provable from $\Pi^1_2\ca + \Sigma^1_3\ind$, and hence also from $\Pi^1_3\ca$. In fact, it is known
that $\Pi^1_3\ca$ already suffices to formalize Davis' proof of $\mathbf{\Sigma^0_3}$ determinacy.

On the other hand, $\Pi^1_2\ca$ by itself can only prove a smaller amount of determinacy:

\begin{theorem}[\cite{heinatsch_mollerfeld}, relying heavily on \cite{mollerfeld_thesis}]\label{thm:det_n_characterization}
Over $\aca$, the theory \[\{\sigmadet{n} : n \in \omega\}\] implies all the $\Pi^1_1$ consequences of
$\Pi^1_2\ca$.
\end{theorem}

In fact, the methods of \cite{mollerfeld_thesis} can be used to prove a strengthening of Theorem \ref{thm:det_n_characterization}. We state the stronger version, accompanied by a proof sketch, as Theorem \ref{thm:det_n_characterization-sharpened} below.

Theorem \ref{thm:det_n_characterization} and a standard argument imply the following corollary,
which is proved exactly as the formally slightly weaker Theorem 5.6 in \cite{medsalem_tanaka_delta_03}:

\begin{corollary}[essentially in \cite{medsalem_tanaka_delta_03}]\label{cor:det_x_lowerbound}
$\Delta^1_3\ca \not \vdash \ouraxiom$.
\end{corollary}

\begin{proof}[Proof sketch] $\Delta^1_3\ca$ is $\Pi^1_4$-conservative over $\Pi^1_2\ca$, so it suffices to show that $\ouraxiom$ is unprovable
in $\Pi^1_2\ca$.

By Theorem \ref{thm:beta}, $\Pi^1_2\ca$ proves the existence of a $\beta_2$-model.
Since every true $\Pi^1_3$ sentence remains true in a $\beta_2$-model,
for every $\Pi^1_3$ sentence $\varphi$ we have:
\[\Pi^1_2\ca \vdash \varphi \Rightarrow \mathrm{Con}(\varphi),\]
where $\mathrm{Con}$ is a formalized consistency statement (cf.\ \cite[Corollary VII.7.8.2]{simpson}). 

Now apply this to $\aca \land \ouraxiom$ as $\varphi$ ($\aca$ is well-known to be axiomatizable by a single $\Pi^1_2$ sentence). If $\varphi$ were provable in $\Pi^1_2\ca$, then also $\mathrm{Con}(\varphi)$ would be provable. But
$\mathrm{Con}(\varphi)$ is $\Pi^0_1$, so by
Theorem \ref{thm:det_n_characterization} it would
follow from $\aca + \sigmadet{n}$ for some $n \in \omega$, and hence from $\varphi$ itself. The theory $\aca + \ouraxiom$ would then contradict G\"odel's second incompleteness theorem.
\end{proof}

Therefore, $\ouraxiom$, and by Theorem \ref{thm:complementation} also the complementation theorem for tree automata, is provable from $\Pi^1_3$- but not $\Delta^1_3$-comprehension. On the other hand,
the argument used to prove Theorem 4.2 in 
\cite{medsalem_tanaka_delta_03} gives:

\begin{theorem}[\cite{medsalem_tanaka_delta_03}]\label{thm:induction}
$\aca + \Delta^1_2\textrm{-}\mathsf{MI}$, and thus also $\Pi^1_2\ca$, proves $\sigmadet{n}$ for all $n \in \omega$, as well as
\[ \forall x\, (\sigmadet{x} \Rightarrow \sigmadet{x+1}).\]
\end{theorem}

\begin{corollary}[\cite{medsalem_tanaka_delta_03}]\label{cor:induction}
$\aca + \Delta^1_2\textrm{-}\mathsf{MI}$, and thus also $\Pi^1_2\ca$, proves $\sigmadet{n}^*$ for all $n \in \omega$, as well as
\[ \forall x\, (\sigmadet{x}^* \Rightarrow \sigmadet{x+1}^*).\]
\end{corollary}
\begin{proof}
This follows immediately from Theorem \ref{thm:induction}, $\sigmadet{x} \Rightarrow \sigmadet{x}^*$, and $\sigmadet{x+1}^* \Rightarrow \sigmadet{x}$.
\end{proof}

We now turn to the positional determinacy of parity games. Of course, by Proposition \ref{prop:bairecantor} and Lemma \ref{lem:dwakropkacztery}, $\ouraxiom$ is equivalent to a special case of the positional determinacy of parity games, so the lower bounds on provability of $\ouraxiom$ remain valid also here: the positional determinacy of all parity games is not provable in $\Delta^1_3\ca$. The lemma below shows that also the upper bound of $\Pi^1_2\ca + \Sigma^1_3\ind$, and a fortiori $\Pi^1_3\ca$, remains valid.

\begin{lemma}\label{lem:parity-induction}
For each $n \in \omega$, $\Pi^1_2\ca$ proves that parity games of index $(0,n)$ are positionally determined. Moreover, $\Pi^1_2\ca$ proves that for every $x$, if parity games of index $(0,x)$ are positionally determined, then so are parity games of index $(0,x+1)$.  
\end{lemma}

\begin{remark}
Our proof of Lemma \ref{lem:parity-induction} is again the formalization of a standard argument in the spirit of the one presented in \cite{thomas}. We strongly believe that the MedSalem-Tanaka argument showing Theorem \ref{thm:induction} can be adapted to give a proof of the step from $(0,x)$ to $(0,x+1)$ in $\aca + \Delta^1_2\textrm{-}\mathsf{MI}$, but we have not attempted to do it.
\end{remark}

\begin{proof}
Since parity games of index $(0,0)$ are trivially positionally determined, it is enough to prove the second part of the statement.
So, assume that all parity games of index $(0,x)$ are positionally determined and let $G = \tuple{V_\exists, V_\forall, v_I, E, \mathrm{rk}}$ be a game of index $(0,x+1)$. For each $v \in V$, we will write $G_v$ to denote a game defined just like $G$ but with starting position $v$ instead of $v_I$.
$\Pi^1_2$-comprehension guarantees the existence of the set
\[W_\forall = \{v: \textrm{ player } \forall \textrm{ has a positional winning strategy in } G_v\}.\]
Let $U$ be $V \setminus W_\forall$. Our aim is to prove that player $\exists$ has a positional winning strategy in $G_v$ for each $v \in U$. Of course, this will in particular imply the positional determinacy of $G$. We first prove a useful claim.

\begin{claim}
Player $\forall$ has a single positional strategy $\sigma_\forall$ which is winning in each $G_v$ for $v \in W_\forall$. Moreover, each play of $G_v$, $v \in W_\forall$, consistent with $\sigma_\forall$ never leaves $W_\forall$.
\end{claim}
To see that the claim is true, note that by Theorem \ref{thm:choice} we can apply
$\Sigma^1_2\textrm{-}\mathsf{AC}$ to conclude that there exists a sequence of positional strategies $\tuple{\sigma_v}_{v \in W_\forall}$ such that $\sigma_v$ is winning for $p$ in $G_v$. $\Pi^1_1$-comprehension implies the existence of a sequence $\tuple{U_v}_{v \in W_\forall}$ such that for each $v$, 
\[U_v = \{w \in W_\forall :  \sigma_w \textrm{ is winning for } \forall \textrm{ in } G_v\}.\]
Observe that $U_v$ is nonempty for each $v \in W_\forall$. Now define $\sigma_\forall(v)$, $v \in W_\forall \cap V_\forall$, to be $\sigma_w(v)$ for $w = \min(U_v)$; for $v \notin W_\forall$, define $\sigma_\forall(v)$ arbitrarily. Given a play $\tuple{v_i}_{i \in \bbn}$ consistent with $\sigma_\forall$ and such that $v_0 \in W_\forall$, we can prove by induction on $i$ that $v_i \in W_\forall$ and that the sequence $\tuple{\min(U_{v_i})}_{i \in \bbn}$ is nonincreasing. Thus, from some round $i$ onwards the play is actually consistent with some fixed strategy $\sigma_w$ which is winning for $\forall$ in $G_{v_i}$. This proves the Claim.

Note that for $v \in U \cap V_\forall$, all edges from $v$ must lead to vertices in $U$. The Claim also implies that for $v \in U \cap V_\exists$, at least one edge from $v$ must lead to a vertex in $U$.

Let $N$ be the set $U \cap \mathrm{rk}^{-1}[\{0\}]$. Consider the following formula $\varphi(x,X)$:
\begin{multline*}
x \in U \land [x \in N \vee x \in X \\
\vee (x \in V_\exists \textrm{ and there exists an edge from } x \textrm{ into }X)\\
\vee (x \in V_\forall \textrm{ and all edges from } x \textrm{ lead into }X).]
\end{multline*}
Since $\varphi$ is $\Delta^1_2$ (in fact, arithmetical) and contains only positive appearances of the variable $X$, Theorem \ref{thm:fixpoint} implies the existence of the inductively generated least fixed point $P$ of the operation  $X \mapsto \{x\in\bbn: \varphi(x,X)\}$, along with a prewellordering $\preccurlyeq$ such that $P$ is the field of $\preccurlyeq$ and $x \in P \Leftrightarrow \varphi(x, \{y: y \prec x\})$. Note that $\exists$ has a positional strategy $\pi$ on $P \setminus N$ such that $\pi$ guarantees reaching a vertex in $N$ in finitely many steps and staying inside $P$ until that happens. The strategy is defined by: $\pi(v) =$ the smallest $w$ such that $w \prec v$ (this is well-defined on $P \setminus N$ unless $N = \emptyset$, in which case $P = \emptyset$ as well).

Now let $R$ be $U \setminus P$. Note that for $v \in R \cap V_\exists$, all edges from $v$ must lead to vertices in $W_\forall \cup R$, while for $v \in R \cap V_\forall$, at least one edge from
$v$ must lead to a vertex in $R$. So, it makes sense to consider the parity game $G{\upharpoonright_R}$ obtained by restricting $G$ to the induced subgraph on $R$. All vertices in $R$ have ranks from 
$\{1,\ldots,x+1\}$, so by the inductive assumption each game $(G{\upharpoonright_R})_v$ for $v \in R$ is positionally determined. It is not hard to check, using the Claim, that a positional winning strategy for $\forall$ in $(G{\upharpoonright_R})_v$ would also give a positional winning strategy for $\forall$ in $G_v$, so in fact it is player $\exists$ who wins each $(G{\upharpoonright_R})_v$. By an obvious analogue of our Claim for $(G{\upharpoonright_R})_v$, this means that $\exists$ has a single positional strategy $\rho$ which wins $(G{\upharpoonright_R})_v$ for each $v \in R$.

Define a positional strategy $\sigma_\exists$ for $\exists$ as follows: 
\[\sigma_\exists(v) = \left\{
                     \begin{array}{ll}
                       \pi(v) & v \in P \setminus N,\\
                       \rho(v) & v \in R,\\
                       \textrm{arbitrary outside } W_\forall & v \in N,\\
                      \textrm{arbitrary} & v\in W_\forall.

                     \end{array}
              \right. \]
It is easy to verify that for each $v \in U$, $\sigma_\exists$ is winning for $\exists$ in $G_v$. \end{proof}

\section{Decidability}\label{sec:decidability}

Theorem \ref{thm:complementation}, combined with the results discussed in Section \ref{sec:determinacy-comprehension}, leads to an easy argument showing that $\Pi^1_2\ca$ is not only unable to prove the complementation theorem for tree automata, but also unable to prove the decidability of the MSO theory of the infinite binary tree by any other method.

\begin{theorem}\label{thm:decidability} 
$\aca$ proves the implications $(1) \Rightarrow (2_n)$, $n \in \omega$, and $\Pi^1_2\ca$ proves the implications $(2_n) \Rightarrow (3)$, $\omega \ni n \ge 3$, where:
\begin{itemize}
\item[$(1)$] all parity games are positionally determined,
\item[$(2_n)$] there is a Turing machine $\mathfrak{t}$ which halts on every input and accepts exactly the $\Pi^1_n$ sentences of MSO true in $(\{0,1\}^{*},S_0,S_1)$,
\item[$(3)$] $\ouraxiom$.
\end{itemize}
\end{theorem}

\begin{remark}
In an ``ideal version'' of Theorem \ref{thm:decidability}, the statements $(2_n)$ would be replaced by a single statement $(2)$: 
``there is a Turing machine $\mathfrak{t}$ which halts on every input and accepts exactly the  sentences of MSO true in $(\{0,1\}^{*},S_0,S_1)$''. However, by Tarski's theorem on the undefinability of truth, the natural way of expressing 
``sentence $\varphi$ is true in $(\{0,1\}^{*},S_0,S_1)$'' as a property of the variable $\varphi$ cannot be formalized in the language of second-order arithmetic. On the other hand, it \emph{is} possible to formulate a truth definition for $\Pi^1_n$ sentences, in particular for $\Pi^1_n$ sentences of MSO; as  Theorem \ref{thm:decidability} shows, the case $n=3$ is crucial here.
\end{remark}

\begin{remark}
With some modifications, the technique used to prove the theorem could also be employed to show that over $\Pi^1_2\ca$, each statement $(2_n)$ for $\omega \ni n \ge 3$ is exactly equivalent to the positional determinacy of almost treelike parity games. However, in the next section we prove a more general result using a different approach.  
\end{remark}

\begin{proof}
Both directions of the proof rely on ideas similar to those in the proof of $(2) \Rightarrow (1)$ of Theorem \ref{thm:complementation}.

$(1)\Rightarrow (2_n)$. By Corollary \ref{cor:complementation}, $(1)$ implies the complementation theorem for tree automata. Moreover, the implication goes through in $\aca$. 

Just like in the proof of $(2) \Rightarrow (1)$ of Theorem \ref{thm:complementation}, complementation for tree automata makes it possible to formalize the usual algorithm constructing an automaton equivalent to an MSO formula $\varphi$ on all labelled binary trees. Given a fixed $n \in \omega$, the correctness of this algorithm  restricted to $\varphi \in \Pi^1_n$ is easily verified in $\aca$.

It remains to have a procedure to deciding whether a given tree automaton $\Aa$ accepts any tree at all. This is equivalent to the problem whether player $\exists$ wins the game $G_\Aa$ described in the proof of Lemma \ref{lem:regular-tree}, and determining the winner of a parity game on a finite arena is a decidable problem provably well within $\aca$ \cite{beckmann-moller:parity}.

$(2_3) \Rightarrow (3)$ For each $x \in \bbn$, there is a $\Pi^1_3$ MSO sentence $\psi_x$ expressing the positional determinacy of treelike parity games of index $(0,x)$. Modulo minor technicalities, $\psi_x$ is ``for all labellings using labels from $\Sigma_{x,0}$, either $W^\exists_{0,x,0}$ or $W^\forall_{0,x,0}$ holds'', where $\Sigma_{x,k}$, $W^\exists_{0,x,k}$, $W^\forall_{0,x,k}$ are as in the proof of  Theorem \ref{thm:complementation}. By Lemma \ref{lem:dwakropkacztery}, $\psi_x$ is true in
$\{0,1\}^{*}$ exactly if $\sigmadet{x}^*$ holds.

Let $\mathfrak{t}$ be a Turing machine given by (2). For each $x$, $\mathfrak{t}(\psi_x) = \mathsf{yes}$ exactly if $\sigmadet{x}^*$ holds. Therefore, by Corollary \ref{thm:induction}, 
$\mathfrak{t}(\psi_0) = \mathsf{yes}$ and for every $x$, $\mathfrak{t}(\psi_x) = \mathsf{yes}$ implies $\mathfrak{t}(\psi_{x+1}) = \mathsf{yes}$.

By arithmetical comprehension, the set $X = \{x: \mathfrak{t}(\psi_x) = \mathsf{yes}\}$ exists. Since $X$ contains $0$ and is closed under successor, we conclude that $X = \bbn$. Therefore each $\psi_x$ is true in $\{0,1\}^{*}$ and so we have
$\ouraxiom$.
\end{proof}

\section{Rabin's theorem as a reflection principle}\label{sec:reflection}

Our aim now is to prove that over $\Pi^1_2\ca$, Rabin's decidability theorem, complementation for tree automata, and the determinacy statements appearing in Theorem \ref{thm:complementation} and Corollary \ref{cor:complementation} are actually all equivalent. Furthermore, all these statements are equivalent to a logical \emph{reflection principle} stating that all $\Pi^1_3$ sentences provable in $\Pi^1_2\ca$ are true.

Our proofs in this section rely in an essential way on both the results and the techniques of M\"ollerfeld's thesis \cite{mollerfeld_thesis}. At present we see no way of avoiding reliance on these advanced techniques.

To achieve this, we first notice that slight changes to the argument of \cite{mollerfeld_thesis} yield a strengthening of Theorem \ref{thm:det_n_characterization}.

\begin{theorem}\label{thm:det_n_characterization-sharpened}
Over $\aca$, the theory \[\{\sigmadet{n} : n \in \omega\}\] axiomatizes the $\Pi^1_3$ consequences of $\Pi^1_2\ca$.
\end{theorem}

\begin{proof}[Proof sketch]
There is a well-known correspondence between sufficiently strong fragments of second-order arithmetic and weak systems of set theory formulated in the usual set-theoretic language $L_\in$ (cf.\ for instance \cite[Chapter VII.3]{simpson}). $L_\in$ is translated into $L_2$ by letting the $\in$-structure of sets be represented by well-founded trees in $\bbn^\bbn$. The inverse translation is the obvious one: $\exists x$ becomes $\exists x \!\in\! \omega^{\mathsf{set}}$, and $\exists X$ becomes $\exists x \!\subseteq\! \omega^{\mathsf{set}}$. (Here we use the symbol  $\omega^{\mathsf{set}}$ for the usual set-theoretic definition of the natural numbers as formalized in $L_\in$. This should not be confused with our use of $\omega$ in the earlier sections.)

$\Pi^1_2\ca$ corresponds to the set theory $\mathsf{Lim}(\preceq_1)$, which consists of the basic axioms known as $\mathsf{BT}^r$ (extensionality, pair, union, $\Delta_0$-separation, and set foundation) and the axiom $\forall x\, \exists y\, [x \in y \wedge y \prec_1 \mathbb{V}]$, where $y \prec_1 \mathbb{V}$ means: ``$y$ is a transitive set, and for all $\bar{a} \in y$ and $\Sigma_1$ formulas $\varphi$ (of $L_\in$), $\varphi(\bar{a})$ iff $y \models \varphi(\bar{a})$''.

Now assume that $\Pi^1_2\ca$ proves the $\Pi^1_3$ sentence $\pi$. It follows that $\mathsf{Lim}(\preceq_1)$ implies the set-theoretic translation $\pi^*$ of $\pi$. As written, $\pi^*$ is a $\Pi_3$ sentence of $L_\in$:
\[\forall x \! \subseteq \! \omega^{\mathsf{set}}
\exists y \! \subseteq \! \omega^{\mathsf{set}} 
\forall z \! \subseteq \! \omega^{\mathsf{set}} \delta(x,y,z),\]
with $\delta$ a bounded formula of $L_\in$. However, 
$\mathsf{Lim}(\preceq_1)$ proves the so-called Axiom $\beta$, which is a variant of the Mostowski collapse lemma: it states that any transitive relation $r$ on a set $a$ can be homomorphically mapped onto the set membership relation on some set $b$. (A statement of Axiom $\beta$, and a sketch of a proof in a subtheory of $\mathsf{Lim}(\preceq_1)$, can be found for instance in \cite[Chapter 11.6]{pohlers}.) Because of this, the $\Pi^1_1$ subformula of $\pi^*$,
$\forall z \! \subseteq \! \omega^{\mathsf{set}} \delta(x,y,z)$, which is equivalent to the statement that a certain relation $r_{x,y}$ parametrized by $x,y$ is well-founded, can provably in $\mathsf{Lim}(\preceq_1)$ be rewritten in a $\Sigma_1$ way: ``there exists a Mostowski collapse of $r_{x,y}$''. Furthermore, the existence of a Mostowski collapse of $r_{x,y}$ implies the well-foundedness of $r_{x,y}$ already
in the weaker theory $\mathsf{BT}^r$ (thanks to the set foundation axiom). What this means is that there is a $\Pi_2$ sentence $\widetilde{\pi^*}$ such that
\begin{align*}
\mathsf{Lim}(\prec_1) & \vdash  \widetilde{\pi^*} \Leftrightarrow \pi^* \\
\mathsf{BT}^r & \vdash \widetilde{\pi^*} \Rightarrow \pi^*. 
\end{align*}
By \cite[Theorem 10.4]{mollerfeld_thesis}, every $\Pi_2$ sentence of $L_\in$ provable in $\mathsf{Lim}(\prec_1)$ is also provable in a certain extension $\mathsf{BT}\sigma^r$ of $\mathsf{BT}^r$. In particular, $\mathsf{BT}\sigma^r$ proves $\widetilde{\pi^*}$ and therefore
also $\pi^*$. However, by \cite[Theorem 8.15]{mollerfeld_thesis}, every $L_2$ sentence whose translation is provable in a theory $\mathsf{Ref} \supseteq \mathsf{BT}\sigma^r$ is itself provable in an arithmetic version of the $\mu$-calculus, which by \cite{heinatsch_mollerfeld} is in turn conservative over $\aca + \{ \sigmadet{n} : n \in \omega\}$.
\end{proof}

Even though Theorem \ref{thm:det_n_characterization-sharpened} concerns relatively strong theories, the work in \cite{heinatsch_mollerfeld,mollerfeld_thesis} involved in its proof relies only on (sketches of) explicit primitive recursive constructions of proofs in various formal theories and, at one point, the cut elimination theorem for first-order logic. Therefore, we have the following corollary of the proof.

\begin{corollary}\label{cor:det_n_characterization-provable}
Primitive recursive arithmetic, and thus a fortiori $\Pi^1_2\ca$, proves that the theory \[\aca + \{\sigmadet{x} : x \in \bbn\}\] axiomatizes the $\Pi^1_3$ consequences of $\Pi^1_2\ca$.
\end{corollary}

The reflection principle $\mathrm{RFN}_{\Pi^1_3}(\Pi^1_2\ca)$ is the formalized statement ``every $\Pi^1_3$ sentence provable in $\Pi^1_2\ca$ is true'', or
\[\forall \varphi \! \in \! \Pi^1_3\, [\mathrm{Pr}_{\Pi^1_2\ca}(\varphi) \Rightarrow \mathrm{Tr}(\varphi)],\]
where $\mathrm{Pr}$ is a standard formalized provability predicate and $\mathrm{Tr}$ is a truth definition for $\Pi^1_3$ sentences. Note that 
$\mathrm{RFN}_{\Pi^1_3}(\Pi^1_2\ca)$ implies $\mathrm{Con}(\Pi^1_2\ca)$ and hence is unprovable in $\Pi^1_2\ca$.

\begin{theorem}\label{thm:reflection}
The following are equivalent over $\Pi^1_2\ca$:
\begin{itemize}
\item[$(1)$] all parity games are positionally determined,
\item[$(2)$] $\ouraxiom$,
\item[$(3)$] for every tree automaton $\Aa$ there exists a tree automaton $\Bb$ such that for any tree $T$, $\Bb$ accepts $T$ exactly if $\Aa$ does not accept $T$,
\item[$(4)$] there is a Turing machine $\mathfrak{t}$ which halts on every input and accepts exactly the $\Pi^1_3$ sentences of MSO true in 
$(\{0,1\}^{<\bbn},S_0,S_1)$,
\item[$(5)$] $\mathrm{RFN}_{\Pi^1_3}(\Pi^1_2\ca)$.
\end{itemize}
\end{theorem}
\begin{proof}
Given Theorems \ref{thm:complementation} and \ref{thm:decidability}, it is enough to prove $(5) \Rightarrow (1)$ and $(2) \Rightarrow (5)$.
We reason in $\Pi^1_2\ca$.

The implication from (5) to (1) is essentially immediate. Let $\varphi(v)$ be a $\Pi^1_3$ formula with free variable $v$ stating that all parity games of index $(0,x)$ are positionally determined. For each $x \in \bbn$, let $\varphi_x$ be the $\Pi^1_3$ sentence obtained by substituting a canonical name (numeral) for $x$ into $\varphi(v)$. By Lemma \ref{lem:parity-induction}, we have $\mathrm{Pr}_{\Pi^1_2\ca}(\varphi_0)$ and 
$\mathrm{Pr}_{\Pi^1_2\ca}(\ulcorner\forall v\, (\varphi(v) \Rightarrow \varphi(v+1))\urcorner)$, so also also $\mathrm{Pr}_{\Pi^1_2\ca}(\varphi_x)$ for each $x \in \bbn$. 
By (5), we obtain $\mathrm{Tr}(\varphi_x)$ for each $x$, so indeed all parity games are positonally determined.

Now assume (2) and consider a $\Pi^1_3$ sentence $\varphi$ of the form \[\forall X\, \exists Y\, \forall Z\, \psi(X,Y,Z),\]
with $\psi$ arithmetical, such that 
$\mathrm{Pr}_{\Pi^1_2\ca}(\varphi_0)$. By Corollary \ref{cor:det_n_characterization-provable}, there is some 
$x \in \bbn$ such that $\forall X\, \exists Y\, \forall Z\, \psi(X,Y,Z)$ is provable from $\aca + \sigmadet{x}$.

Assume that $\varphi$ is false. This means that there is some set $S$ such that $\forall Y\, \exists Z\, \neg \psi(S,Y,Z)$. By Theorem \ref{thm:beta}, there is a countable coded $\beta_2$-model $M$ such that $S \in M$. Since $M$ is a $\beta_2$-model, we have
\[M \models \forall Y\, \exists Z\, \neg \psi(S,Y,Z).\]
However, since both $\aca$ and $\sigmadet{x}$ are true and have complexity no higher than $\Pi^1_3$, we also have
\[M \models \aca \land \sigmadet{x}.\]

Now, take the proof $\pi$ 
of $\forall X\, \exists Y\, \forall Z\, \psi(X,Y,Z)$ from 
$\aca \land \sigmadet{x}$. Since $M$ is countable, we can express
the statement ``the $y$-th formula in $\pi$ is true in $M$'' as an
arithmetical formula $\tau(y)$. By induction on $y$, $\tau(y)$ holds
for all $y$ smaller than the length of $\pi$. But the last formula in $\pi$ is 
$\forall X\, \exists Y\, \forall Z\, \psi(X,Y,Z)$,
so
\[M \models \forall X\, \exists Y\, \forall Z\, \psi(X,Y,Z),\]
a contradiction. This proves (5).
\end{proof}

\section{Further work}

\begin{description}
\item[Determinacy] As part of Theorem \ref{thm:reflection}, we prove the equivalence of the determinacy principle $\ouraxiom$ and positional determinacy of all parity games. Our proof of this equivalence relies on the advanced techniques of \cite{mollerfeld_thesis}, goes through an intermediate step involving the reflection principle $\mathrm{RFN}_{\Pi^1_3}(\Pi^1_2\ca)$, and requires $\Pi^1_2\ca$ as a base theory. 

Is there a simpler, more elementary proof of this equivalence with $\Pi^1_2\ca$ as the base theory? Is there a proof in $\aca$? As a possibly easier question, does either of the implications in Corollary \ref{cor:complementation} reverse over $\aca$?
\item[Büchi's theorem] Büchi's Theorem states that nondeterministic \emph{word} automata are closed under complementation (this implies the decidability of the MSO theory of $(\bbn, \le)$). Because $\aca$ proves that word automata can be determinized (cf.\ the proof of Theorem \ref{thm:complementation}), also Büchi's Theorem is provable within $\aca$. However, what is the exact reverse-mathematical strength of the theorem?
\item[Weak MSO over the binary tree] Can statements related to the decidability of the \emph{weak} MSO theory of the infinite binary tree be characterized as determinacy principles, for instance as the determinacy of Boolean combinations of open games? 

\item[MSO theory over the reals ${\mathbb R}$] In \cite{shelah} it shown that the MSO theory of the real line with the natural ordering is undecidable. However, once the quantification is restricted to $F_\sigma$ sets, the theory is decidable \cite{rabin_dec}. Is there a proof of this decidability result not relying on the full strength of Rabin's theorem? 
\end{description}


\bibliographystyle{abbrv}

\bibliography{biblio}

\end{document}